\documentclass[reqno, 12pt]{amsart}

\numberwithin{equation}{section}

\usepackage{amsfonts,amssymb,amsmath,amsthm}
\usepackage{enumerate}
\usepackage{bbm}
\usepackage{times}
\usepackage{amsfonts}
\usepackage{amssymb}
\usepackage{bbm}
\usepackage{times}
\usepackage{amsfonts}
\usepackage{amssymb}
\usepackage{bbm}
\usepackage{times}
\usepackage{tabularx}
\usepackage{pgfplots}

\usepackage{enumerate}
\usepackage{amsmath}
\usepackage{amsthm}
\usepackage{color}

\makeatletter
\@namedef{subjclassname@2020}{%
  \textup{2020} Mathematics Subject Classification}
\makeatother

\newtheorem{theor}{Theorem}[section]

\newtheorem{limma}[theor]{Lemma}

\newtheorem{prop}[theor]{Proposition}

\newcounter{other}            

\def \D{\mathbb{D}}

\def \N{\mathbb{N}}

\def \H {\mathcal H}

\def  \bm {\boldsymbol}
\def \H{\mathcal H}
\def \P {\mathcal P}

\begin{document}
\title[The norm of the Hilbert matrix operator   on  Bergman  spaces]
{ The norm of the Hilbert matrix operator   on  Bergman  spaces}

\author{Guanlong Bao, Liu Tian and Hasi Wulan}
\address{Guanlong Bao\\
Department of Mathematics\\
    Shantou University\\
    Shantou, Guangdong 515821, China}
\email{glbao@stu.edu.cn}

\address{Liu Tian\\
Department of Mathematics\\
    Shantou University\\
    Shantou, Guangdong 515821, China}
\email{20ltian@stu.edu.cn}

\address{Hasi Wulan\\
Department of Mathematics\\
    Shantou University\\
    Shantou, Guangdong 515821,  China}
\email{wulan@stu.edu.cn}


\begin{abstract}
  Karapetrovi\'c conjectured that the norm of the Hilbert matrix operator  on the Bergman space $A^p_\alpha$
is equal to $\pi/\sin((2+\alpha)\pi/p)$ when $-1<\alpha<p-2$.  In this paper, we provide a proof of this conjecture
for  $0\leq  \alpha \leq \frac{6p^3-29p^2+17p-2+2p\sqrt{6p^2-11p+4}}{(3p-1)^2}$, and this range of $\alpha$  improves the best known result when  $\alpha>\frac{1}{47}$ and $\alpha \not=1$.
\end{abstract}

\maketitle

\section{Introduction}

In 1894, the infinite matrix $\bm {H}=((n+k+1)^{-1})_{n, k\in \N}$ was introduced by Hilbert  \cite{Hil} in  approximation theory. Here   $\N$ is the set of nonnegative integers.
This matrix has been well known as the Hilbert matrix, which has applications across numerous fields, including linear algebra, numerical analysis, and operator theory  (cf. \cite{Ka, M,  SL}).
  The Hilbert matrix is also  a typical example of Hankel matrices whose $n$, $k$ entry is a function of $n+k$ (cf. \cite{Wi}). We  refer to monographs \cite[Chapter 1 and  Chapter 10]{Pe} and  \cite[Chapter 5]{Ha} for the theory associated with the Hilbert matrix.

The Hilbert matrix $\bm {H}$ can be regarded as an operator acting on sequence spaces. More precisely,
given a complex sequence $\{a_n\}_{n\in \N}$, consider
$$
\mathrm{H}:\ \ \{a_n\}_{n\in \N} \mapsto \left\{\sum_{k=0}^\infty \frac{a_k}{n+k+1}\right\}_{n\in \N}.
$$
For $p>1$, it is known  that the operator $\mathrm{H}$ is bounded on the sequence space $\ell^p$ and its operator norm
is equal to $\pi / \sin(\pi/p)$; see for instance \cite[Theorem 286 and Theorem 323]{HLP}.

Let $\text{Hol}(\D)$ be the space of functions analytic in the open unit disk $\D$ of the  complex plane.  For $f(z)=\sum_{n=0}^\infty  a_n z^n$ in $\text{Hol}(\D)$,
by the action of $\bm H$ on the sequence of the  Taylor coefficients of $f$,    the Hilbert matrix operator $\H$  is given  by
\begin{equation}\label{Hdefine}
\H(f)(z)=\sum_{n=0}^\infty \left(\sum_{k=0}^\infty  \frac{a_k}{n+k+1}\right)z^n, \ \ z\in \D,
\end{equation}
whenever the right-hand side of \eqref{Hdefine}  makes sense for every  $z\in \D$ and defines an analytic function in $\D$.

In 2000,  Diamantopoulos and Siskakis \cite{DS} initiated the study of the Hilbert matrix operator  on spaces of  functions analytic in $\D$.
They  proved that  the Hilbert matrix operator $\H$ is bounded  on the Hardy space $H^p$ for $p>1$, and got an upper bound of the corresponding
operator norm. Later,  Dostani\'{c}, Jevti\'{c} and  Vukoti\'{c} \cite{DJV} obtained that this operator  norm on $H^p$ is equal to $\pi / \sin(\pi/p)$, which coincides with the norm of the operator
$\mathrm{H}$ on the sequence space $\ell^p$. \L{}anucha, Nowak and Pavlovi\'c \cite{LNP} investigated the operator  $\H$ on some spaces of analytic functions.
Jevti\'{c} and Karapetrovi\'{c} \cite{JK1} considered  $\H$ on more general Besov spaces. Lindstr\"{o}m, Miihkinen and Wikman \cite{LMW2} and Dai \cite{Da1} studied the norm of $\H$ on Korenblum spaces.
Lindstr\"{o}m, Miihkinen and Norrbo \cite{LMN} got the exact value of the essential norm of $\H$ on some classical analytic function spaces. Ye and Feng \cite{YF} investigated the norm of $\H$ from logarithmically weighted Bergman spaces to the classical Bergman spaces. See \cite{BDS} and \cite{K0} for the surveys  of the Hilbert matrix operator on spaces of analytic functions.

  For $0<p<\infty$ and $\alpha>-1$, recall that the
Bergman space $A^p_\alpha$ consists of functions $f$ in $\text{Hol}(\D)$ such that
$$
\|f\|_{A^p_\alpha}=\left(\int_\D |f(z)|^p dm_\alpha (z)\right)^{\frac 1 p}<\infty,
$$
where $dm_\alpha (z)=(\alpha+1)(1-|z|^2)^{\alpha} dm(z)$. Here and below, $dm$ is the area measure on $\D$ normalized so that $m(\D)=1$; that is,
$$
dm(z)=\frac{1}{\pi} dxdy=\frac{1}{\pi} rdrd\theta, \ \ z=x+iy=re^{i\theta}.
$$
If $\alpha=0$, we write $A^p$ instead of $A^p_\alpha$.
See \cite{HKZ, Z} for  $A^p_\alpha$ spaces.

There are also  a good number of mathematicians who have investigated the boundedness and the  norm of the  Hilbert matrix operator $\H$  on Bergman spaces.
In 2004,   Diamantopoulos \cite{Di} proved that  $\H$ is bounded on $A^p$ if and only if $p>2$,  and estimated the upper bound of this operator norm.
For $p\geq 4$, in 2008  Dostani\'{c}, Jevti\'{c} and  Vukoti\'{c} \cite{DJV} obtained  that the exact value of the  norm of $\H$ on $A^p$  is $\pi/\sin(2\pi/p)$.  For $2<p<4$, in 2018 Bo\u{z}in and  Karapetrovi\'{c} \cite{BK} proved that the  norm of $\H$ on $A^p$  is also $\pi/\sin(2\pi/p)$; see Lindstr\"{o}m, Miihkinen and Wikman \cite{LMW2} for another proof of this result.  For the action of $\H$  on $A^p_\alpha$,   we  recall some  previous results  as follows.
\begin{enumerate}[(i)]
  \item   In 2017, Jevti\'{c} and   Karapetrovi\'{c} \cite{JK} studied the  Hilbert matrix operator on a family of spaces of analytic functions.  In particular, they showed that  $\H$  is bounded on $A^p_\alpha$ if and only if  $1<\alpha+2<p$.
  See Galanopoulos, Girela, Pel\'aez and Siskakis \cite{GGPS} for the study of a generalized Hilbert operator on $A^p_\alpha$  spaces.
  \item  In 2018, Karapetrovi\'{c} \cite{K2} got the lower bound $\pi/\sin((2+\alpha)\pi/p)$ of the norm of  $\H$ on $A^p_\alpha$ for   $1<\alpha+2<p$, and  proved that when $\alpha\geq 0$ and $p\geq 2(\alpha+2)$ this lower bound is the exact value of the  norm of $\H$ on $A^p_\alpha$. In \cite[p. 516]{K2},  Karapetrovi\'{c}  conjectured that  if $1<\alpha+2<p$, then the  norm of $\H$ on $A^p_\alpha$ is $\pi/\sin((2+\alpha)\pi/p)$.
       \item In 2021,  Lindstr\"{o}m, Miihkinen  and Wikman \cite{LMW1} proved Karapetrovi\'{c}'s  conjecture  for $\alpha> 0$ and
$\alpha+2+\sqrt{\alpha^2+\frac{7}{2}\alpha+3}\leq p<2(\alpha+2)$.
  \item In 2021, Karapetrovi\'{c} \cite{K1} proved his conjecture for $\alpha>0$ and $p\geq \alpha_0$, where $\alpha_0$ is a unique zero of the function
  $$\Phi_\alpha(x)=2x^2-(4(\alpha+2)+1)x+2\sqrt{\alpha+2}\sqrt{x}+\alpha+2 $$
  on the interval $(\alpha+2, 2(\alpha+2))$.
 Using this result, he confirmed his conjecture when  $\alpha>0$ and
$$\alpha+2+\sqrt{(\alpha+2)^2-\left(\sqrt{2}-\frac{1}{2}\right)(\alpha+2)}\leq p<2(\alpha+2),$$
which improves the conclusion in (iii).  For $-1<\alpha<0$ and $p>\alpha+2$, Karapetrovi\'{c} \cite{K1}  also gave an  upper bound of the norm of  $\H$ on $A^p_\alpha$, which  was further studied by
Bralovi\'c and Karapetrovi\'{c} \cite{BKa}, and Norrbo \cite{N}.
  \item In 2023, Dmitrovi\'c and Karapetrovi\'{c} \cite{DK} proved Karapetrovi\'{c}'s  conjecture for $\alpha>0$ and $\frac{3\alpha}{4}+2+\sqrt{\left(\frac{3\alpha}{4}+2\right)^2-\frac{\alpha+2}{2}}\leq p$.
 They pointed out that this improves the best previously known result for all $\alpha>\frac{1}{2}$.
  \item In 2024, Dai \cite{Da}  demonstrated Karapetrovi\'{c}'s  conjecture  when $p$ and $\alpha$ satisfy one of the following conditions:
\begin{itemize}
  \item $\alpha>0$ and $2+\frac{3\alpha}{4}+\frac{1}{4}\sqrt{9\alpha^2+40\alpha+48}\leq p<2(\alpha+2)$;
  \item  $\alpha>0$,  $\alpha+2<p<2+\frac{3\alpha}{4}+\frac{1}{4}\sqrt{9\alpha^2+40\alpha+48}$,   and a nontrivial  condition of $p$ and $\alpha$ involving a double integration and the Beta function;
  \item $0<\alpha\leq \frac{1}{47}$ and $p>\alpha+2$;
  \item $\alpha=1$ and $p>3$;
  \item $-1<\alpha<0$ and $p\geq 2(\alpha+2)$.
\end{itemize}
  \end{enumerate}

The value $\frac{3\alpha}{4}+2+\sqrt{\left(\frac{3\alpha}{4}+2\right)^2-\frac{\alpha+2}{2}}$ in (v) is equal to the number  $2+\frac{3\alpha}{4}+\frac{1}{4}\sqrt{9\alpha^2+40\alpha+48}$ in (vi), which was
found  independently by the authors in \cite{DK} and \cite{Da}.

In summary, when $\alpha>0$ and $\alpha+2<p<2(\alpha+2)$, Karapetrovi\'{c}'s  conjecture was not   fully settled. Karapetrovi\'{c}'s  conjecture  is also open for  $-1<\alpha<0$ and $\alpha+2<p<2(\alpha+2)$.

The known results of  Karapetrovi\'{c}'s  conjecture, which were mentioned earlier,  were stated based on  $p(\alpha)$-curves.  Certainly, they can be rewritten   based on  $\alpha(p)$-curves.   In this paper, we give  a   proof of Karapetrovi\'{c}'s  conjecture for
\begin{align}\label{Z::main inequ}
0\leq  \alpha \leq \frac{6p^3-29p^2+17p-2+2p\sqrt{6p^2-11p+4}}{(3p-1)^2},
\end{align}
which is  based on $\alpha(p)$-curves.  Throughout this paper, when  the value $\frac{6p^3-29p^2+17p-2+2p\sqrt{6p^2-11p+4}}{(3p-1)^2}$ appears,  we emphasize that
we always  assume this value is reasonable; that is $6p^2-11p+4\geq 0$ and $p \not = 1/3$. Then we will see that condition \eqref{Z::main inequ} also yields $p>\alpha+2$.
 As recalled before, Karapetrovi\'{c}'s  conjecture for the cases  $0\leq \alpha\leq \frac{1}{47}$ and $\alpha=1$ was  thoroughly
answered in the literature. In Section 4, we will point out that our range of $\alpha$ in \eqref{Z::main inequ} is  the best so far when $\alpha>\frac{1}{47}$ and $\alpha \not=1$.

\section{A  result based on $\alpha(p)$-curves}

In this section, we prove   the following conclusion  which   is stated based on $\alpha(p)$-curves and  will be used to prove our main result.

\begin{prop}\label{1main}
 Let
\begin{align}\label{ABC:456}
&\max \left\{\frac{6p^3-29p^2+17p-2-2p\sqrt{6p^2-11p+4}}{(3p-1)^2}, \ 0 \right\} \nonumber  \\
\leq & \ \alpha \ \leq\frac{6p^3-29p^2+17p-2+2p\sqrt{6p^2-11p+4}}{(3p-1)^2}.
\end{align}
Then
$$
\|\mathcal{H}\|_{A_\alpha^p}=\frac{\pi}{\sin\frac{(\alpha+2)\pi}{p}},
$$
where  $\|\mathcal{H}\|_{A_\alpha^p}$  is the norm of the Hilbert matrix operator $\H$ on $A_\alpha^p$.
 \end{prop}

\vspace{0.1truecm}
\noindent {\bf  Remark 1.}\ \
In the last section, we will see that condition \eqref{ABC:456} in Proposition \ref{1main}  can be replaced by the weaker condition \eqref{Z::main inequ}.
Recall that  $\H$  is bounded on $A^p_\alpha$ if and only if  $1<\alpha+2<p$. We remark that condition  \eqref{Z::main inequ}  implies $p>\alpha+2$. In fact,
condition \eqref{Z::main inequ} gives
\begin{align}\label{TL:AAA}
2p\sqrt{6p^2-11p+4}\geq -6p^3+29p^2-17p+2.
\end{align}
If $p\leq 0$, then the left-hand side of \eqref{TL:AAA} is non-positive and the right-hand side of \eqref{TL:AAA} is positive, which is impossible.
 In other words,  condition  \eqref{Z::main inequ}  yields $p>0$. When $p>0$,  a direct computation gives
$$\frac{6p^3-29p^2+17p-2+2p\sqrt{6p^2-11p+4}}{(3p-1)^2}<p-2. $$
Combining this with \eqref{Z::main inequ}, we get $p>\alpha+2$.

Denote by  $\P$ the set of polynomials.  For $1<\alpha+2<p$, recall  that there exists a positive constant $C$ such that
 \begin{align}\label{ZZ::ZZ}
\|\H(g)\|_{A_\alpha^p}\leq C \|g\|_{A_\alpha^p}
\end{align}
for all $g \in \P$.  In fact,   \eqref{ZZ::ZZ} yields that   for any $f(z)=\sum_{n=0}^\infty  a_n z^n$  in $A^p_\alpha$, the right-hand side of \eqref{Hdefine}  makes sense for every  $z\in \D$ and defines an analytic function in $\D$. This can be done similarly to \cite[p. 1069]{Di}. See \cite[Lemma 2.3]{BS} for a similar conclusion about a generalized Hilbert operator on certain Banach spaces  of analytic functions.

  The following conclusion  is from  \cite{K2}, which is a reason for Karapetrovi\'{c} to give his conjecture.
\begin{theor}\cite[Theorem 1.1]{K2} \label{1ower bound}
If $1<\alpha+2 <p$, then
$$
\|\mathcal{H}\|_{A_\alpha^p}\geq\frac{\pi}{\sin\frac{(\alpha+2)\pi}{p}}.
$$
 \end{theor}

For  $p>\alpha+2>0$, it is known that
 \begin{align}\label{Z1}
 \int_{0}^1\frac{t^{\frac{\alpha+2}{p}-1}}{(1-t)^{\frac{\alpha+2}{p}}}\,dt= B\left(\frac{\alpha+2}{p}, 1-\frac{\alpha+2}{p}\right)=\frac{\pi}{\sin\frac{(\alpha+2)\pi}{p}},
\end{align}
 where $B$ is the Beta function. Then, for $p>\alpha+2>0$, the function
 \begin{align}\label{F1}
F_{p,\alpha}(r)=r^{p-4-\frac{3\alpha}{2}}
\int_0^{\frac{2r}{1+r}}\frac{t^{\frac{\alpha+2}{p}-1}}{(1-t)^{\frac{\alpha+2}{p}}}
\,dt - \int_{0}^1\frac{t^{\frac{\alpha+2}{p}-1}}{(1-t)^{\frac{\alpha+2}{p}}}\,dt
\end{align}
is well defined for  $r\in (0,1]$.

The  following property of the function $F_{p,\alpha}$ is useful to prove Proposition \ref{1main}.

 \begin{limma}\label{auxiliary}
Let $p>\alpha+2>0$  and
\begin{align}\label{A formula}
&\frac{6p^3-29p^2+17p-2-2p\sqrt{6p^2-11p+4}}{(3p-1)^2} \nonumber  \\
\leq & \ \alpha \ \leq\frac{6p^3-29p^2+17p-2+2p\sqrt{6p^2-11p+4}}{(3p-1)^2}.
\end{align}
Then $F_{p,\alpha}(r)\leq0$ for all $r\in(0,1]$.
 \end{limma}
\begin{proof}
 For   $r\in(0,1)$,  $F'_{p,\alpha}(r)=r^{p-5-\frac{3\alpha}{2}}g(r)$, where
$$
g(r)=\left(p-4-\frac{3\alpha}{2}\right)\int_0^{\frac{2r}{1+r}}\frac{t^{\frac{\alpha+2}{p}-1}}{(1-t)^{\frac{\alpha+2}{p}}}
\,dt +\frac{(2r)^{\frac{\alpha+2}{p}}}{(1+r) (1-r)^{\frac{\alpha+2}{p}}}.
$$
It is easy to get
\begin{eqnarray*}
g'(r)=\frac{2^{\frac{\alpha+2}{p}}r^{\frac{\alpha+2}{p}-1} }{(1+r)^2 (1-r)^{\frac{\alpha+2}{p}+1}}k(r).
\end{eqnarray*}
Here
\begin{eqnarray*}
k(r)
=\left(\frac{3\alpha }{2}  +5-p\right)r^2+\left(  \frac{\alpha+2}{p} -1\right)r+\frac{\alpha+2}{p}+ p-\frac{3\alpha }{2}-4.
\end{eqnarray*}
Then
\begin{eqnarray*}
k'(r)=\left (3\alpha  +10-2p\right)r + \frac{\alpha+2}{p} -1.
\end{eqnarray*}

Now, suppose   $3\alpha +10-2p=0$. Note that $p>\alpha+2>0$. Then
\begin{eqnarray*}
k(r)=\left(  \frac{\alpha+2}{p} -1\right)r+\left(\frac{\alpha+2}{p}+ p-\frac{3\alpha }{2}-4\right)
\end{eqnarray*}
is a decreasing function on $[0, 1]$ and $k(1)=\frac{2(\alpha+2)}{p}>0$. Hence $k(r)>0$ for $r\in [0, 1]$.

On the other hand, let $3\alpha  +10-2p\neq0$. Then $k$ is  a quadratic function. Denote   by $\Delta_k$ the discriminant for $k(r)=0$. Then
$$
\Delta_k=\left(  \frac{\alpha+2}{p} -1\right)^2-4\left(\frac{3\alpha }{2}  +5-p\right)\left(\frac{\alpha+2}{p}+ p-\frac{3\alpha }{2}-4\right).
$$
Since     $6p^2-11p+4\geq 0$ and  $p\not= \frac{1}{3}$,  condition \eqref{A formula} holds if and only if
$$
(3p-1)^2\alpha^2+(-12p^3+58p^2-34p+4)\alpha+(4p^4-36p^3+89p^2-44p+4)\leq0,
$$
which is equivalent to    $\Delta_k \leq 0$.  Combining this with $k(1)=\frac{2(\alpha+2)}{p}>0$, we get  $k(r)\geq 0$ for $r\in [0, 1]$.

Thus, $k(r)$ is always nonnegative on $[0, 1]$. This gives  $g'(r)\geq 0$ for $r\in (0,1)$. Then $g$ is increasing on $(0,1)$. Bear in mind that
$g(0)=0$ and $g$ is right-continuous at 0. Hence $g(r)\geq 0$ when  $r\in (0,1)$.  Consequently, $F'_{p,\alpha}(r)\geq 0$ for  $r\in (0,1)$.
So, $F_{p,\alpha}$ is  increasing in $(0,1)$. Also, $F_{p,\alpha}$ is left-continuous at 1 and $F_{p,\alpha}(1)=0$. Hence
 $F_{p,\alpha}(r)\leq0$ for all $r\in(0,1]$.
\end{proof}

Let $1<\alpha+2<p$.  Recall that the Hilbert matrix operator  $\mathcal{H}$ on $A^p_\alpha$ has  an integral representation by  weighted composition
 operators $T_t$ (cf. \cite{Di, K2}):
$$
\mathcal{H}(f)(z)=\int_0^1 T_t(f)(z)\,dt,\ \ f\in A^p_\alpha,
$$
where $T_t(f)(z)=\omega_t(z)f(\phi_t(z))$ and
$$
\omega_t(z)=\frac{1}{1-(1-t)z},  \ \  \phi_t(z)=\frac{t}{1-(1-t)z}.
$$

Next we give the proof of Proposition  \ref{1main}.

\begin{proof}[Proof of Proposition  \ref{1main}]
Let $f\in A^p_\alpha$. By Remark 1, $p>\alpha+2\geq 2$. Using  the Minkowski inequality, it is known  from \cite[pp. 520-521]{K2}  that
\begin{align}\label{D1}
\|\mathcal{H}(f)\|_{A_\alpha^p}\leq\int_0^1\|T_t(f)\|_{A_\alpha^p}\,dt,
\end{align}
and
$$
\|T_t(f)\|_{A_\alpha^p}= \frac{t^{\frac{\alpha+2}{p}-1}}{(1-t)^{\frac{\alpha+2}{p}}}\left((\alpha+1)\int_{D_t}(g_t(z))^\alpha|f(z)|^p|z|^{p-4}\,dm(z)\right)^{\frac{1}{p}}.
$$
Here
$$
g_t(z)=\frac{2\mathrm{Re}z-(2-t)|z|^2-t}{(1-t)|z|^2}>0, \ \ z\in D_t,
$$
and $D_t=\mathbb{D}\left(\frac{1}{2-t},\frac{1-t}{2-t}\right)$ is the Euclidean disk  of radius $(1-t)/(2-t)$ and center $1/(2-t)$.
Note that
$$
g_t(z)\leq\frac{2|z|-(2-t)|z|^2-t}{(1-t)|z|^2}=\frac{1}{|z|^2}\left(1-|z|^2-\frac{(1-|z| )^2}{1-t}\right).
$$
Because of  $\alpha\geq0$, we get
$$
(g_t(z))^\alpha\leq\frac{1}{|z|^{2\alpha}}\left(1-|z|^2-\frac{(1-|z| )^2}{1-t}\right)^\alpha.
$$
Consequently,
\begin{align}
&||T_t(f)||_{A_\alpha^p} \leq \frac{t^{\frac{\alpha+2}{p}-1}}{(1-t)^{\frac{\alpha+2}{p}}} \times \nonumber \\
&\small{\left((\alpha+1)\int_{D_t}\left(1-|z|^2-\frac{(1-|z|)^2}{1-t}\right)^\alpha|f(z)|^p|z|^{p-4-2\alpha}\,dm(z)\right)^{\frac{1}{p}}.} \label{C1}
\end{align}
Clearly,
\begin{align}\label{A:::BC}
D_t\subseteq\left\{z\in\mathbb{D}:\frac{t}{2-t}<|z|<1\right\}.
\end{align}
Write
$$
K_{p,\alpha}(t)=2(\alpha+1)
\int_{\frac{t}{2-t}}^1\left(1-r^2-\frac{(1-r)^2}{1-t}\right)^\alpha r^{p-3-2\alpha}M_p^p(r,f)\,dr,
$$
where
$$
M_p(r,f)=\left(\frac{1}{2\pi} \int_0^{2\pi}|f(re^{i\theta})|^p\  d\theta\right)^{\frac{1}{p}}.
$$
Then it follows from  \eqref{C1} that
$$
\|T_t(f)\|_{A_\alpha^p}\leq \frac{t^{\frac{\alpha+2}{p}-1}}{(1-t)^{\frac{\alpha+2}{p}}}(K_{p,\alpha}(t))^{\frac{1}{p}}.
$$
This and \eqref{D1} yield
\begin{align}\label{C2}
\|\mathcal{H}(f)\|_{A_\alpha^p}\leq\int_0^1 \frac{t^{\frac{\alpha+2}{p}-1}}{(1-t)^{\frac{\alpha+2}{p}}}(K_{p,\alpha}(t))^{\frac{1}{p}}\,dt.
\end{align}

Next,  because $p>\alpha+2\geq 2$,  we consider
\begin{align}\label{C6}
&\int_0^1 \frac{t^{\frac{\alpha+2}{p}-1}}{(1-t)^{\frac{\alpha+2}{p}}}\left((K_{p,\alpha}(t))^{\frac{1}{p}}-\|f\|_{A_\alpha^p}\right) \,dt \nonumber \\
=& \int_0^1 \frac{t^{\frac{\alpha+2}{p}-1}}{(1-t)^{\frac{\alpha+2}{p}}} \left((K_{p,\alpha}(t))^{\frac{1}{p}}-\left(J_{p,\alpha}\right)^{\frac{1}{p}}\right) \,dt,
\end{align}
where
$$
J_{p,\alpha}=2(\alpha+1)\int_{0}^1M_p^p(r,f)(1-r^2)^\alpha r\,dr.
$$
Recall that the following elementary inequality (cf. \cite[p. 39]{HLP})
$$
x^\beta-y^\beta\leq\beta y^{\beta-1}(x-y)
$$
holds when  $x>0$, $y>0$ and $\beta\in(0,1)$. Then \eqref{C6} gives
\begin{align}\label{X1}
&\int_0^1 \frac{t^{\frac{\alpha+2}{p}-1}}{(1-t)^{\frac{\alpha+2}{p}}}\left((K_{p,\alpha}(t))^{\frac{1}{p}}-\|f\|_{A_\alpha^p}\right) \,dt \nonumber \\
\leq & \ \frac{1}{p} \left(J_{p,\alpha}\right)^{\frac{1}{p}-1} \int_0^1 \frac{t^{\frac{\alpha+2}{p}-1}}{(1-t)^{\frac{\alpha+2}{p}}} \left(K_{p,\alpha}(t)- J_{p,\alpha} \right)     \,dt.
\end{align}
By  $\alpha\geq0$ and  $(1+r)^{\alpha}\geq 2^\alpha r^{\alpha/2}$, we obtain
\begin{align*}
& K_{p,\alpha}(t)- J_{p,\alpha} \nonumber \\
\leq & \ 2(\alpha+1)\int_{\frac{t}{2-t}}^1\left(1-r^2-(1-r)^2\right)^\alpha r^{p-3-2\alpha}M_p^p(r,f)\,dr \nonumber \\
&- 2^{\alpha+1}(\alpha+1)\int_{0}^1M_p^p(r,f)(1-r)^\alpha r^{\frac{\alpha}{2}+1}\,dr \nonumber \\
= & \ 2^{\alpha+1}(\alpha+1)\int_{\frac{t}{2-t}}^1\left(1-r\right)^\alpha r^{p-3-\alpha}M_p^p(r,f)\,dr \nonumber \\
&- 2^{\alpha+1}(\alpha+1)\int_{0}^1M_p^p(r,f)(1-r)^\alpha r^{\frac{\alpha}{2}+1}\,dr.
\end{align*}
Thus,
\begin{align*}
&\int_0^1 \frac{t^{\frac{\alpha+2}{p}-1}}{(1-t)^{\frac{\alpha+2}{p}}} \left(K_{p,\alpha}(t)- J_{p,\alpha} \right)     \,dt \nonumber \\
\leq &\  2^{\alpha+1}(\alpha+1)\int_0^1 \frac{t^{\frac{\alpha+2}{p}-1}}{(1-t)^{\frac{\alpha+2}{p}}}\times \nonumber \\
&\left(\int_{\frac{t}{2-t}}^1\left(1-r\right)^\alpha r^{p-3-\alpha}M_p^p(r,f)\,dr-\int_{0}^1M_p^p(r,f)(1-r)^\alpha r^{\frac{\alpha}{2}+1}\,dr\right) \,dt.
\end{align*}
This together with the Fubini theorem gives
\begin{align*}
&\int_0^1 \frac{t^{\frac{\alpha+2}{p}-1}}{(1-t)^{\frac{\alpha+2}{p}}} \left(K_{p,\alpha}(t)- J_{p,\alpha} \right)     \,dt \nonumber \\
\leq& \ 2^{\alpha+1}(\alpha+1) \int_0^1  M_p^p(r,f) (1-r)^\alpha r^{\frac{\alpha}{2}+1}  F_{p,\alpha}(r) \, dr,
\end{align*}
where $F_{p,\alpha}(r)$ is defined as \eqref{F1}. From
Lemma \ref{auxiliary}, we get
\begin{align}\label{X2}
\int_0^1 \frac{t^{\frac{\alpha+2}{p}-1}}{(1-t)^{\frac{\alpha+2}{p}}} \left(K_{p,\alpha}(t)- J_{p,\alpha} \right)     \,dt \leq 0.
\end{align}
Thus, \eqref{Z1},    \eqref{X1},  and  \eqref{X2} yield
\begin{align*}
 \int_0^1 \frac{t^{\frac{\alpha+2}{p}-1}}{(1-t)^{\frac{\alpha+2}{p}}}(K_{p,\alpha}(t))^{\frac{1}{p}}\,dt
\leq \frac{\pi}{\sin\frac{(\alpha+2)\pi}{p}} \|f\|_{A_\alpha^p}.
\end{align*}
Combining this with \eqref{C2}, we obtain
\begin{align*}
\|\mathcal{H}(f)\|_{A_\alpha^p}\leq\frac{\pi}{\sin\frac{(\alpha+2)\pi}{p}} \|f\|_{A_\alpha^p}
\end{align*}
for all $f\in A_\alpha^p$, which gives
$$
\|\mathcal{H}\|_{A_\alpha^p}\leq \frac{\pi}{\sin\frac{(\alpha+2)\pi}{p}}.
$$
By this and Theorem \ref{1ower bound}, we get the desired result. The proof of Proposition   \ref{1main} is complete.
\end{proof}
\vspace{0.1truecm}
\noindent {\bf  Remark 2.}\ \  The proof of Proposition  \ref{1main} follows closely  that   of Theorem 3.2 in \cite{Da}. For the completeness, we provide all details here.
 But  we use the  more natural estimate \eqref{A:::BC} which  is different from that in \cite{Da}. This  yields  two  new functions $K_{p,\alpha}$ and  $F_{p,\alpha}$ in the proof of
 Proposition  \ref{1main}, and a new estimate of $F_{p,\alpha}$ in  Lemma \ref{auxiliary}. Hence, we get new ranges of parameters $p$ and $\alpha$,  for which  Karapetrovi\'{c}'s  conjecture holds.

\section{A result based on a  $p(\alpha)$-curve}

From the $\alpha(p)$-curve
$$\alpha=\frac{6p^3-29p^2+17p-2+2p\sqrt{6p^2-11p+4}}{(3p-1)^2}$$  in Proposition  \ref{1main}, it is not easy to get the corresponding $p(\alpha)$-curve by solving  the equation above.
In this section, applying the proof of Proposition  \ref{1main}, we  obtain  Proposition  \ref{2main} which is stated based on  a $p(\alpha)$-curve.   Proposition  \ref{2main} improves
the best previously known result in some sense, and can also be  used in the next section.

\begin{prop}\label{2main}
 Let $\alpha\geq0$ and $p\geq\frac{9+3\alpha +\sqrt{9\alpha^2+30\alpha+33} }{4}$. Then
$$
\|\mathcal{H}\|_{A_\alpha^p}=\frac{\pi}{\sin\frac{(\alpha+2)\pi}{p}}.
$$
 \end{prop}
\begin{proof}
From  the proof of Proposition  \ref{1main}, it suffices to prove $F_{p,\alpha}(r)\leq0$ for all $r\in(0,1]$ when $\alpha\geq0$ and
$$ p\geq\frac{9+3\alpha +\sqrt{9\alpha^2+30\alpha+33} }{4}.$$
Clearly, the  assumptions of this proposition   imply   $p>\alpha+2$.

Use the same symbols as in  the proof of  Lemma \ref{auxiliary}. Recall that
\begin{eqnarray*}
k(r)
=\left(\frac{3\alpha }{2}  +5-p\right)r^2+\left(  \frac{\alpha+2}{p} -1\right)r+\frac{\alpha+2}{p}+ p-\frac{3\alpha }{2}-4,
\end{eqnarray*}
and
\begin{eqnarray*}
k'(r)=\left (3\alpha  +10-2p\right)r + \frac{\alpha+2}{p} -1.
\end{eqnarray*}

When   $3\alpha +10-2p=0$,  we have proven that  $k(r)>0$ for $r\in [0, 1]$ in Lemma \ref{auxiliary}. Now let $3\alpha  +10-2p\neq0$.
Then $k$ is  a quadratic function.

 {\sl {Case 1}}.   Consider
 \begin{align}\label{YY1}
 p\geq \frac{9+3\alpha +\sqrt{9\alpha^2+62\alpha+97}}{4}.
 \end{align}
Note that    $p>0$ and $\alpha\geq 0$. Then \eqref{YY1}   is equivalent to $k'(1)\leq0$. Also, $k'(0)< 0$. Thus, $k'(r)\leq0$ for $r\in [0, 1]$,
which means that $k$ is a decreasing function  in $[0,1]$. Note that $k(1)=\frac{2(\alpha+2)}{p}>0$.  Hence $k(r)>0$ for $r\in [0, 1]$.

{\sl {Case 2}}.  Let
\begin{align}\label{YY2}
\frac{9+3\alpha +\sqrt{9\alpha^2+30\alpha+33} }{4}  \leq p<\frac{9+3\alpha +\sqrt{9\alpha^2+62\alpha+97} }{4}.
\end{align}
Bear also  in mind  that    $p>\alpha+2$ and $\alpha\geq 0$.  With these ranges of $p$ and $\alpha$,  \eqref{YY2} holds if and only if
\begin{align}\label{YY3}
\begin{cases}
\frac{\alpha+2}{p}+ \frac{2p}{3}-\alpha-3\geq0, \\
k'(1)=9+3\alpha+\frac{\alpha+2}{p}-2 p>0.
\end{cases}
\end{align}
The  assumptions of this proposition  also  gives
$
p\geq \frac {3\alpha}{2}+3;
$
that is,
$$
\frac{\alpha+2}{p}+ p-\frac{3\alpha }{2}-4 \geq \frac{\alpha+2}{p}+\frac{2p}{3}-\alpha-3.
$$
Hence  the first inequality   in \eqref{YY3}  implies
\begin{align}\label{YY4}
k(0)=\frac{\alpha+2}{p}+ p-\frac{3\alpha }{2}-4\geq0.
\end{align}
By \eqref{YY3}, \eqref{YY4}, and $p>\alpha+2$,  we see that
\begin{eqnarray*}
\Delta_k
&=&\left(  \frac{\alpha+2}{p} -1\right)^2-2\left(3\alpha +10-2p\right)\left(\frac{\alpha+2}{p}+ p-\frac{3\alpha }{2}-4\right)
\\
&\leq&\left(  \frac{\alpha+2}{p} -1\right)^2+2\left(\frac{\alpha+2}{p}-1\right)\left(\frac{\alpha+2}{p}+ p-\frac{3\alpha }{2}-4\right)
\\
&=&\left(  \frac{\alpha+2}{p} -1\right)\left(\frac{3(\alpha+2)}{p}+ 2p-3\alpha -9\right)\leq 0.
\end{eqnarray*}
From this and  $k(1)=\frac{2(\alpha+2)}{p}>0$, one gets   $k(r)\geq 0$ for $r\in [0, 1]$.

Thus, with the assumptions, we always have  $k(r)\geq 0$ for $r\in [0, 1]$. Next, following the proof of
Lemma \ref{auxiliary}, we obtain $F_{p,\alpha}(r)\leq0$ for all $r\in(0,1]$. The proof  is complete.
\end{proof}

 We finish this section by   some  remarks.

\vspace{0.1truecm}
\noindent {\bf  Remark 3.}\ \
Proposition  \ref{1main} and Proposition  \ref{2main} cannot imply each other. For instance,   $\alpha=9$ and $p=17$ satisfy conditions in Proposition  \ref{1main}, yet they fail to
satisfy conditions  in Proposition \ref{2main}. On the other hand,  $\alpha=9$ and $p=20$ meet the assumptions  of Proposition  \ref{2main}, but they do not satisfy those  in Proposition \ref{1main}.

\vspace{0.1truecm}
\noindent {\bf  Remark 4.}\ \  As stated  in Section 1,  there have been  several results proving Karapetrovi\'{c}'s  conjecture when $\alpha>0$ and $p\geq \alpha_1$, where $\alpha_1$ is a  number in $(\alpha+2, 2(\alpha+2))$. Next we show  that for all $\alpha\geq\frac{-1+\sqrt{13}}{6}$ and $\alpha \not=1$, the value $\frac{9+3\alpha +\sqrt{9\alpha^2+30\alpha+33} }{4}$ in  Proposition \ref{2main}  improves the best previously known result.

For $\alpha>0$,  it is easy  to see that
\begin{align*}
\frac{9+3\alpha +\sqrt{9\alpha^2+30\alpha+33} }{4} &<\ 2+\frac{3\alpha}{4}+\frac{1}{4}\sqrt{9\alpha^2+40\alpha+48}  \\
&=\  \frac{3\alpha}{4}+2+\sqrt{\left(\frac{3\alpha}{4}+2\right)^2-\frac{\alpha+2}{2}}. \nonumber
\end{align*}
Recall that for  $0<\alpha\leq \frac{1}{47}$ and $\alpha=1$, Karapetrovi\'{c}'s  conjecture was proved in \cite{Da}.
Consequently, for $\alpha>\frac{1}{47}$ and $\alpha\not =1$,    Proposition \ref{2main} improves the corresponding results in  (v) and (vi) stated in Section 1.

Now we compare Proposition  \ref{2main} with the corresponding conclusion in  (iv) mentioned in  Section 1.
Let $\alpha>0$. By  \cite[p. 5918]{K1},
 the  unique zero $\alpha_0$ of the function $\Phi_\alpha$ on $(\alpha+2, 2(\alpha+2))$ satisfies
\begin{align*}
&\alpha+2+\sqrt{(\alpha+2)^2-(\alpha+2)}\\
<& \ \alpha_0 <\ \alpha+2+\sqrt{(\alpha+2)^2-\left(\sqrt{2}-\frac{1}{2}\right)(\alpha+2)}.
 \end{align*}
Note that
\begin{align}\label{GG-1}
 \frac{9+3\alpha +\sqrt{9\alpha^2+30\alpha+33} }{4}\leq \alpha+2+\sqrt{(\alpha+2)^2-(\alpha+2)}
 \end{align}
 if and only if
\begin{align}\label{GG-2}
 \sqrt{9\alpha^2+30\alpha+33}\leq \alpha-1+4\sqrt{\alpha^2+3\alpha+2}.
\end{align}
 If $\alpha \geq 1$, it is clear that \eqref{GG-2} holds, so does  \eqref{GG-1}. If $0<\alpha<1$, then \eqref{GG-2} is equivalent to
 $$ 0<1-\alpha+\sqrt{9\alpha^2+30\alpha+33}\leq 4\sqrt{\alpha^2+3\alpha+2}; $$
that is,
$$
3\alpha^3+7\alpha^2+ \alpha-2=3(\alpha+2)\left(\alpha-\frac{-1+\sqrt{13}}{6}\right)\left(\alpha-\frac{-1-\sqrt{13}}{6}\right)\geq 0.
$$
Thus,   if $\alpha\geq \frac{-1+\sqrt{13}}{6}$, then \eqref{GG-1} holds and hence the value $\frac{9+3\alpha +\sqrt{9\alpha^2+30\alpha+33} }{4}$ in  Proposition  \ref{2main} is smaller than
the   zero $\alpha_0$ of the function $\Phi_\alpha$.

In conclusion, for all $\alpha\geq \frac{-1+\sqrt{13}}{6}$ and $\alpha \not=1$, the value $\frac{9+3\alpha +\sqrt{9\alpha^2+30\alpha+33} }{4}$ in  Proposition   \ref{2main} improves the best previously known result.

\vspace{0.1truecm}
\noindent {\bf  Remark 5.}\ \ From Theorem 3.2 in \cite{Da} and Corollary  1.3 in \cite{K2},  Karapetrovi\'{c}'s  conjecture holds if   $\alpha>0$ and  either of the following conditions holds:
\begin{itemize}
  \item
  \begin{align}\label{Z-D0}
  p \geq 2+\frac{3\alpha}{4}+\frac{1}{4}\sqrt{9\alpha^2+40\alpha+48};
  \end{align}
  \item $\alpha+2<p<2+\frac{3\alpha}{4}+\frac{1}{4}\sqrt{9\alpha^2+40\alpha+48}$
and
\begin{align}\label{Z-D}
&\int_0^{1}\frac{t^{\frac{\alpha+2}{p}-1}}{(1-t)^{\frac{\alpha+2}{p}}}\left(\int_{t^2}^1(1-r)^\alpha r^{p-3-\alpha}\,dr\right)\,dt  \nonumber \\
 \leq & \ B\left(1+\alpha,2+\frac{\alpha}{2}\right) B\left(\frac{\alpha+2}{p},1-\frac{\alpha+2}{p}\right).
\end{align}
\end{itemize}
It is not easy to use condition \eqref{Z-D} to  detect the values $p$ and $\alpha$ for  Karapetrovi\'{c}'s  conjecture.
Proposition   \ref{2main} means that Karapetrovi\'{c}'s  conjecture holds when  $\alpha>0$ and
$$
\frac{9+3\alpha +\sqrt{9\alpha^2+30\alpha+33} }{4}\leq p<2+\frac{3\alpha}{4}+\frac{1}{4}\sqrt{9\alpha^2+40\alpha+48},
$$
which is not clear from  \eqref{Z-D}.

Corresponding to  Remark 4 and Remark 5, we can also understand Proposition  \ref{2main}   by the following figure.

\begin{figure*}[htbp]
\includegraphics[width=12cm,height=7cm]{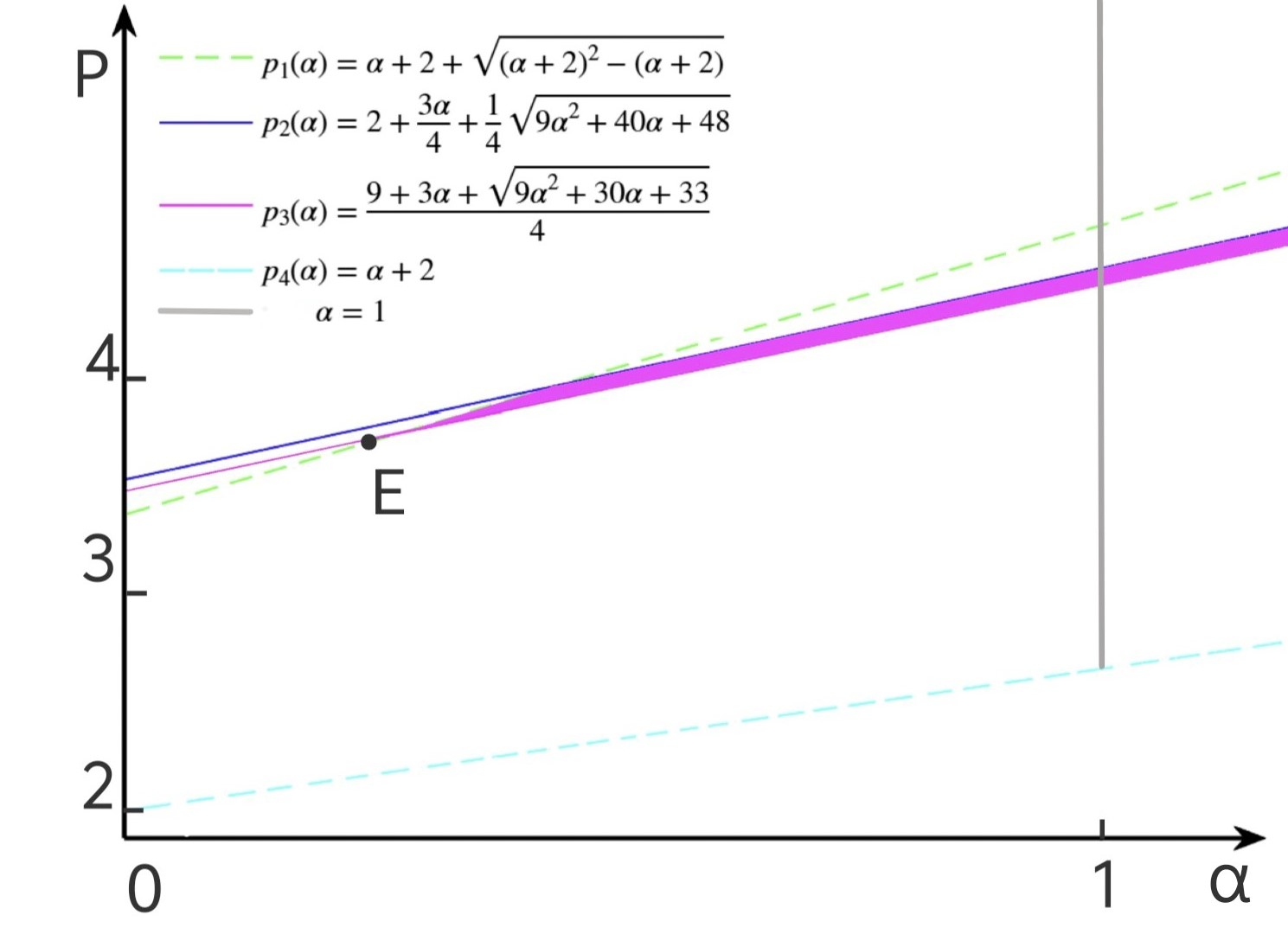}
   \caption{}
\end{figure*}
\noindent Proposition  \ref{2main}  confirms   Karapetrovi\'{c}'s  conjecture   when $(\alpha, p)$ is in  the purple  region in Figure 1, which was  not thoroughly resolved   from previous results.
Note that the point
{\small \begin{align*}
E \left(\frac{-1+\sqrt{13}}{6}, \frac{11+\sqrt{13}}{6}+\sqrt{\left(\frac{11+\sqrt{13}}{6}\right)^2-\frac{11+\sqrt{13}}{6}}\right)
\end{align*}}
is the intersection of curves $p_1(\alpha)$ and $p_3(\alpha)$.
When  $\alpha\geq \frac{-1+\sqrt{13}}{6}$ and $\alpha\not=1$, compared to the previous curves,   the curve $p_3(\alpha)=\frac{9+3\alpha +\sqrt{9\alpha^2+30\alpha+33} }{4}$ derived from  Proposition  \ref{2main}
is the closest to the  straight line $p_4(\alpha)=\alpha+2$.  Thus,  the value $\frac{9+3\alpha +\sqrt{9\alpha^2+30\alpha+33} }{4}$ in Proposition  \ref{2main}  is better than previous results  from the literature when  $\alpha\geq \frac{-1+\sqrt{13}}{6}$ and $\alpha\not=1$.

\section{The main result}

In this section, using  Proposition
\ref{1main}  and  Proposition \ref{2main},  we confirm   Karapetrovi\'{c}'s  conjecture when $p$ and $\alpha$ satisfy  \eqref{Z::main inequ}.  This conclusion is better than Proposition
\ref{1main}  and  Proposition \ref{2main}. In particular, the  conclusion  is  the best so far when $\alpha>\frac{1}{47}$ and $\alpha \not=1$.

   Let  $\alpha \geq 0$ and $p> \alpha+2$. Then   $$p\geq\frac{9+3\alpha +\sqrt{9\alpha^2+30\alpha+33} }{4}$$ in Proposition  \ref{2main} can be rewritten as
$$
\alpha \leq \frac{2p^2-9p+6}{3p-3};
$$
 condition $$\alpha+2+\sqrt{(\alpha+2)^2-\left(\sqrt{2}-\frac{1}{2}\right)(\alpha+2)}\leq p$$ in (iv) of  Section 1 holds if and only if
$$
\alpha \leq  \frac{p^2-4p+2\sqrt{2}-1}{2p+\frac 1 2-\sqrt{2}};
$$
and  condition
$$2+\frac{3\alpha}{4}+\frac{1}{4}\sqrt{9\alpha^2+40\alpha+48}\leq p$$ in (v) or (vi) of  Section 1 is equivalent to
$$\alpha \leq \frac{2p^2-8p+2}{3p-1}. $$

Thus  the previous lower bounds of $p$ can be translated to upper bounds for $\alpha$.  Figure 1 is a $p(\alpha)$-graph.  By these upper bounds  for $\alpha$  and \eqref{ABC:456}, we give an $\alpha(p)$-graph below
from which one can  see how well Proposition  \ref{2main} fills in the gap created by the lower bound in Proposition  \ref{1main}.

\begin{figure*}[htbp]
\includegraphics[width=10cm,height=6cm]{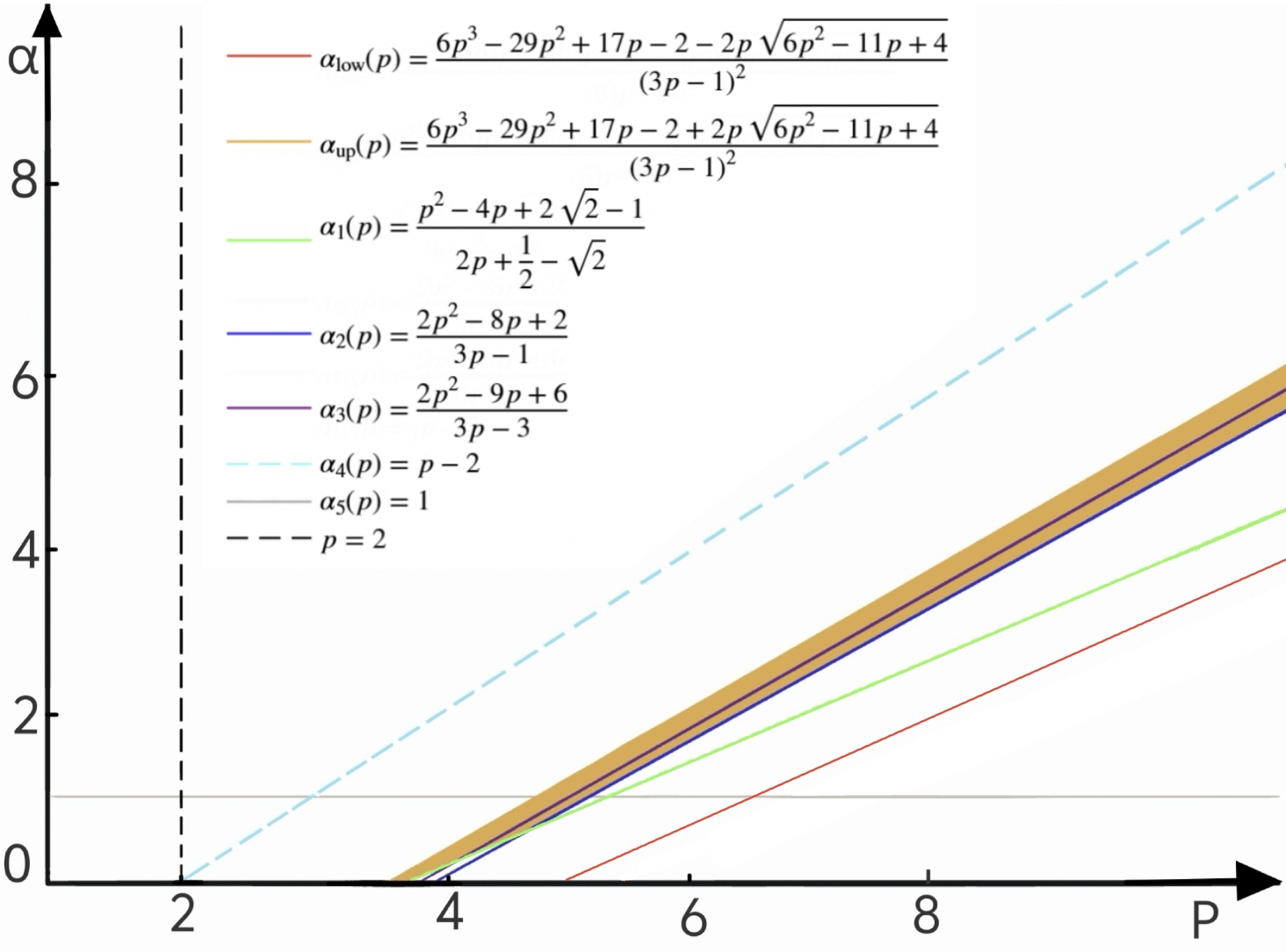}
   \caption{}
\end{figure*}

Proposition  \ref{1main}  proves Karapetrovi\'{c}'s  conjecture   when $(\alpha, p)$ is in  the yellow region in Figure 2, which was  not thoroughly resolved  from previous results.

Note that $\alpha_{up}$ is defined in  Figure 2.  When
$$
\max\left\{\frac{2p^2-9p+6}{3p-3}, \ \frac{p^2-4p+2\sqrt{2}-1}{2p+\frac 1 2-\sqrt{2}}, \ \frac{2p^2-8p+2}{3p-1} \right\} \geq \alpha \geq 0,
$$
a direct computation  yields
{\small \begin{align}\label{ZZZZ:::1}
\alpha_{up}(p) \geq \max\left\{\frac{2p^2-9p+6}{3p-3}, \ \frac{p^2-4p+2\sqrt{2}-1}{2p+\frac 1 2-\sqrt{2}}, \ \frac{2p^2-8p+2}{3p-1} \right\},
\end{align} }
which  is consistent with Figure 2. In fact,
  Figure 2 also gives us the information that
  Karapetrovi\'{c}'s  conjecture holds when  $0\leq \alpha \leq \alpha_{up}(p)$.
 Using Proposition \ref{1main} and Proposition  \ref{2main}, we will  give a proof of this conclusion, which is the main result in this paper.
 From  \eqref{ZZZZ:::1}, since  Karapetrovi\'{c}'s  conjecture holds when $0\leq \alpha\leq \frac{1}{47}$ and $\alpha=1$, we see that
 the following theorem appears to be the best known   so far,  when $\alpha>\frac{1}{47}$ and $\alpha \not=1$.

\begin{theor}\label{corr}
Suppose
\small {
\begin{align*}
0\leq \alpha \leq \frac{6p^3-29p^2+17p-2+2p\sqrt{6p^2-11p+4}}{(3p-1)^2}.
\end{align*} } Then
$$
\|\mathcal{H}\|_{A_\alpha^p}=\frac{\pi}{\sin\frac{(\alpha+2)\pi}{p}}.
$$
\end{theor}
\begin{proof}
By Remark 1,  we get $p>\alpha+2$, which yields the boundedness of $\H$ on  $A_\alpha^p$.
Recall that
$$\alpha_{low}(p)=\frac{6p^3-29p^2+17p-2-2p\sqrt{6p^2-11p+4}}{(3p-1)^2}.$$
When $p>2$, a calculation  gives
\begin{align}\label{TL::}
\frac{2p^2-9p+6}{3p-3} > \alpha_{low}(p).
\end{align}
From Proposition   \ref{1main}, the desired result holds if
$$\max\{\alpha_{low}(p), 0\} \leq \alpha \leq \alpha_{up}(p).$$
Then it suffices to prove the conclusion when $\alpha_{low}(p)>0$.  For such case, \eqref{TL::} gives $\frac{2p^2-9p+6}{3p-3} >0$.
By Proposition   \ref{2main}, Karapetrovi\'{c}'s  conjecture holds if
$
0 \leq \alpha \leq \frac{2p^2-9p+6}{3p-3}.
$
Combining this with \eqref{TL::}, we get that Karapetrovi\'{c}'s  conjecture holds when $\alpha_{low}(p)>0$ and
$ 0\leq \alpha \leq \alpha_{low}(p)$. Also, by Proposition   \ref{1main}, Karapetrovi\'{c}'s  conjecture holds if  $\alpha_{low}(p)>0$ and
$\alpha_{low}(p)\leq \alpha \leq \alpha_{up}(p)$. Thus, we get the desired  conclusion when $\alpha_{low}(p)>0$. The proof is complete.
\end{proof}

\noindent{\text{\bf Acknowledgements.}}
 The work was supported  by National Natural Science Foundation of China (No. 12271328).

\end{document}